\documentclass[amsmath,amssymb]{amsart}
 \usepackage{amsmath}

\newtheorem{corollary}{Corollary}

\newtheorem{proposition}{Proposition}

\theoremstyle{definition}
\newtheorem{definition}{Definition}
\newtheorem{remark}{Remark}

\newtheorem{example}{Example}

\newcommand{\R}        {\mathbb R}

\title{Lecture Notes: The Galerkin Method \footnote{Delivered as part of a summer reading course on Numerical Methods for Integral Equations at Simon Fraser University}}

\author[Raghavendra Venkatraman]{}

\email{raghav.16venkat@gmail.com}

\begin{document}

\maketitle

\centerline{\scshape Raghavendra Venkatraman \footnote{School of Mechanical Engineering, Indian Institute of Technology Roorkee, India. \\Present Address: Department of Mathematics, Simon Fraser University, Burnaby BC.}}

\medskip

In these notes, we consider the \textit{analysis} of Galerkin Method and its application to computing approximate solutions of integral equations. The emphasis is on Galerkin methods with an orthogonal basis. We introduce the Galerkin method in the framework of a Hilbert space. We give a computational example that illustrates the importance of choosing the right basis for the approximating finite dimensional subspaces. We then consider the solution of an integral equation whose exact solution is known, and present a sample matlab code to illustrate the success of the method. Finally, we give an interesting proof for how the Galerkin Method can be use to study the existence of solutions to a nonlinear boundary value problem based on its variational formulation. 

\section{Warming up: Some Analysis, Geometry and Hand Calculations} 
In this section, we motivate the Galerkin Method from analysis, and give some geometrical intuition for the case of symmetric problems. We subsequently pass on to a simple example, that illustrates the importance of choosing a good basis with the help of a numerical example. Most of the material of this section is based on the Numerical Analysis book by K. Atkinson \cite{atkinson}.\\
We begin by introducing a Hilbert Space $V,$ and a bilinear form $a(\cdot, \cdot)  : V \times V \to \R$ which is both bounded and $V-$elliptic, that is, 
\begin{definition}[Boundedness of Bilinear Form] The form $a(\cdot,\cdot)$ is bounded if there exists a positive number $M>0$ such that 
\[ |a(u,v)| \leq M\|u\|_V \|v\|_V \hspace{1cm} \forall u,v \in V, \]
\end{definition}
and 
\begin{definition}[V-Ellipticity] The form $a(\cdot,\cdot)$ is $V-$ elliptic provided there exists a constant $\alpha$ such that 
\[ a(v,v) \geq \alpha \|v\|_V^2 \hspace{1cm} \forall v \in V. \]
\end{definition}
Now, the basic problem of interest is the following: Given a functional $l \in V',$ (the dual space of $V,$) we look to find $u \in V$ for which 
\begin{equation} \label{prob1} a(u,v) = l(v) \hspace{1cm} \forall v \in V. \end{equation}
The Lax-Milgram Lemma \cite{kesavan, atkinson} gaurantees the existence of a unique solution $u$ to the problem \eqref{prob1}. In case of partial differential or integral equations, the space $V$ is infinite dimensional, hence it is rarely possible to find an exact solution to problem \eqref{prob1}. \\

In order to construct an approximate solution, it is natural to consider a finite dimensional approximation to \eqref{prob1}. For instance, we consider the finite dimensional subspace $V_N \subset V,$ an $N-$dimensional subspace of the space $V,$ and project the problem \eqref{prob1} onto $V_N,$ that is, we seek for 
\begin{equation} \label{prob2} u_N \in V_N, \hspace{1cm} a(u_N,v) = l(v), \hspace{1cm} \forall v \in V_N, \end{equation}
and hope that the resulting sequence of approximate solutions converges in some sense to the solution of the original problem.
Since $a(\cdot,\cdot)$ is bounded and $V-$ elliptic as before, Lax-Milgram grants the existence of a unique solution $u_N \in V_N$ for $l \in V'.$
We rewrite \eqref{prob2} as a linear system as follows. Supposing $\{ \phi_i\}_{i=1}^N$ be a basis of the finite dimensional subspace $V_N,$ we write, 
\[ u_N = \sum_{j=1}^N \xi_j \phi_j, \]
and take $v \in V_N$ to be the basis functions $\phi_i.$ Accordingly, we get the linear system
\begin{equation} A \xi = b, \end{equation}
where $\xi \in \R^N$ is the unknown vector, $[A]_{i,j} = [a(\phi_i,\phi_j)] \in \mathcal{M}^N(\R)$ (motivated from mechanics) is termed the \textit{stiffness matrix} and $b = (l(\phi_i)) \in \R^N$ the \textit{loading vector}. \\
The approximate solution $u_N$ differs from the exact solution $u,$ to increase the accuracy of the solution, we look for solutions in larger finite dimensional subspaces. Accordingly, corresponding to a sequence of subspaces $V_1 \subset V_2  \subset \cdots \subset V, \dim{V_j} < \infty, j = 1,2, \cdots ,$ we compute the approximate solutions $u_i \in V_i, i = 1,2,\cdots,$ this solution procedure generates the so-called Galerkin Method. We now consider the special case when $a(u,v)$ is symmetric, that is, 
\[ a(u,v) = a(v,u) \hspace{1cm} \forall u,v \in V.\]
In this case, it can be checked readily that the \eqref{prob1} is equivalent to the folowing minimization programme, 
\begin{equation} u \in V, E(u) = \inf_{v \in V} E(v), \hspace{1cm} E(v) = \frac{1}{2}a(v,v) - l(v). \end{equation}
Correspondingly, the approximate problem \eqref{prob2} can be viewed as a minimization programme over the finite dimensional subspace $V_N$ of $V,$
\begin{equation}
\mbox{ Seek } u_N \in V_N,  \hspace{1cm} E(u_N) = \inf_{v \in V_N} E(v). 
\end{equation}
Before proceeding, we consider the following example. 
\begin{example} Consider the two point boundary value problem, 
\begin{equation} 
 \begin{array}{l}
  -u'' = f  \mbox{ in } (0,1), \\
u(0) = u(1) = 0. 
\end{array}
\end{equation}
The weak formulation of the problem, is
\begin{equation}
\mbox{ Seek } u \in V, \hspace{1cm} \int_0^1 u'v'\,dx = \int_0^1 fv \,dx \hspace{1cm} \forall v \in V, 
\end{equation}
where $V = H^1_0(0,1).$ Accordingly, Lax-Milgram grants the existance of a unique solution. To develop a Galerkin method, we identify a finite dimensional subspace of basis functions that satisfy the boundary conditions at $x=0,1.$ A natural choice for the subspace is to consider one whose basis vectors satisfy the boundary conditions of the bvp: 
\[ V_N = \mbox{ span }\{x^i (1-x)| i = 1,2,\cdots N \}. \]
Approximating,
\[ u_N = \sum_{j=1}^N \xi_j x^j(1-x); \]
where the coefficients $\xi_j$ are determined by the Galerkin equations
\[ \int_0^1 u_N' v' \,dx = \int_0^1 fv \,dx \hspace{1cm} \forall v \in V_N. \]
Taking $v$ to be each of the basis functions $x^i(1-x), 1 \leq i \leq N,$ we obtain the following linear system of equations, 
\[A \xi = b, \]
where $\xi = (\xi_1, \xi_2, \cdots, \xi_N)^T$ is the vector of unknowns, $b \in \R^N$ is the vector whose $i$-th component is $\int_0^1 f(x)x^i(1-x)\,dx.$ The coefficient matrix $A,$ has $(i,j)$th entry given by, 
\begin{equation} \nonumber 
\int_0^1 [x^j(1-x)]'[x^i(1-x)]'\,dx = \frac{(i+1)(j+1)}{i+j+1} + \frac{(i+2)(j+2)}{i+j+3} - \frac{(i+1)(j+2) + (i+2)(j+1)}{i+j+2}. \end{equation}
The coefficient matrix is fairly ill-conditioned, and as a matter of fact, using the $\mathrm{cond(X,p)}$ command of matlab, it is fairly easy to calculate the condition numbers of the matrices $A,$ for instance in $2-$ norm, 

\begin{equation} \nonumber 
\begin{array}{cc}
N & Condition Number (A)\\
\hline
3 & 891.6637\\
4 & 2.4233e+04\\
5 & 6.5617e+05\\
6 & 1.7919e+07\\
7 & 4.9532e+08\\
8 & 1.3867e+10\\
9 & 3.9288e+11\\
10& 1.1282e+13\\
\hline
\end{array}
\end{equation}

We now work out the previous problem now using a different finite dimensional subspace, 
\[ V_N = span\{ \sin(i\pi x) | i = 1, \cdots , N\}.\]
The basis functions are \textit{orthogonal} with respect to the inner product defined by the bilinear form, 
\[
\int_0^1 (\sin j\pi x)' (\sin i \pi x)' \,dx = \frac{ij\pi^2}{2}\delta_{ij}. \]
Again, writing 
\[ u_N = \sum_{j=1}^N \xi_j \sin j\pi x, \]

the coefficients $\xi_j$ satisfy the following linear system, which is in fact \textit{diagonal}, hence we get the solution, 
\[ \xi_j = \frac{2}{\pi^2i^2} \int_0^1 f(x) \sin i \pi x\,dx,  \hspace{1cm} i = 1 \cdots N. 
\]
In fact, the Galerkin solution can be written in the form of a Kernel Approximation, 
\[ u_N = \int_0^1 f(t) K_N(x,t)\,dt, \]
where the kernel function 
\[ K_N(x,t) = \frac{2}{\pi^2} \sum_{i=1}^N \frac{\sin j\pi x \cdot \sin j \pi t}{j^2}. \]
\end{example}

\begin{remark} The above examples clearly illustrate that it is very important to choose appropriate basis functions to the finite dimensional subspaces. 
\end{remark}
\begin{remark} The Galerkin Method is not just a numerical scheme for approximating solutions to a differential or integral equations. By passing to the limit, we can even prove some existence results. This point will be illustrated via a simple nonlinear example towards the end of the lecture. 
\end{remark}
We now prove a result that serves as the basis for proving error estimates and convergence. 
\begin{proposition}[Cea Inequality] \label{prop1} Let $V$ a Hilbert Space, $V_N \subset V$ is a subspace, $a(\cdot, \cdot)$ is a bounded and $V-$elliptic bilinear form on $V,$ and $l \in V'.$ Let $u \in V$ be the solution of the problem \eqref{prob1}, and $u_N \in V_N$ be the Galerkin Approximation defined in \eqref{prob2}. Then 
\begin{equation} \|u - u_N\|_V \leq C \inf_{v \in V_N} \|u-v\|_V. \end{equation}
\end{proposition}
\begin{proof} Subtracting \eqref{prob2} and \eqref{prob1} with $v \in V_N,$ we obtain the orthogonality relation, 
\begin{equation} a(u-u_N, v) = 0, \hspace{1cm} \forall v \in V_N. \end{equation}
Accordingly, using $V-$ ellipticity, and boundedness of $a,$  we have for any $v \in V_N,$
\begin{equation} 
\begin{array}{cc}
\alpha \|u-u_N\|_V^2 & \leq a(u-u_N, u-u_N)\\
& = a(u-u_N,u-v)\\
& \leq M\|u-u_N\|_V \|u-v\|_V
\end{array}
\end{equation}

Accordingly, 
\[ \|u - u_N\|_V \leq c\|u-v\|_V,\] 
since $v \in V_N$ is arbitrary, the inequality follows. 
\end{proof}

We have the following convergence result as a consequence of the Cea Inequality, 
\begin{corollary}[Convergence] Making the assumptions stated in \ref{prop1}, Assume $V_1 \subset V_2 \subset \cdots $ is a sequence of finite dimensional subspaces of $V$ with the property 
\begin{equation} 
\overline{\bigcup_{n\geq 1} V_n} = V. \end{equation}
Then the Galerkin method converges, 
\begin{equation} \|u-u_N\|_V \to 0 \hspace{1cm} \mbox{ as } n \to \infty, 
\end{equation}
where $u_N \in V_N$ is the Galerkin Solution. 
\end{corollary}
\section{The Galerkin Method applied to Integral Equations}
We begin with the linear integral equation, 
\begin{equation} \label{2.1} 
\lambda u(x) - \int_\Omega k(x,y)u(y)\,dy = f(x), \hspace{1cm} x \in \Omega,
\end{equation}
where we usually seek solutions living in atleast a complete function space $V.$ 
If $u_n$ is an approximate solution living in a finite dimensional subspace of $V$, we define the residual in the approximation by 
\begin{definition}[Residual] 
\begin{equation} \label{2.2}
r_n \equiv (\lambda - K)u_n - f.\end{equation}
\end{definition}
Accordingly, we consider $V= L^2(\Omega),$ a Hilbert Space, and denote by $(\cdot, \cdot)$ its inner product. In the Galerkin methods, we require $r_n$ to satisfy
\begin{equation}  \label{2.3}
(r_n,  \phi_i) = 0, \hspace{1cm} i = 1, 2, \cdots \kappa_n,
\end{equation}
where $\kappa_n$ is the dimension of the approximating subspace of $V = L^2(\Omega).$
Notice that the left hand side of \eqref{2.3} is the Fourier coefficient $r_n$ associated with the basis functions $\phi_i.$ In particular, when $V$ is a separable Hilbert space, and $\{\phi_1, \cdots, \phi_\kappa\}$ consists of the leading members of the orthonormal family $\Phi= \{\phi_i\}_{i\geq1} $ which spans $V.$ This implies that \eqref{2.3} demands that the leading terms in the Fourier expansion of $r_n$ with respect to $\Phi$ vanish. 
To find $u_n,$ we use the above approximation in conjunction with the equation $(\lambda - K)u = f. $ We obtain the linear system 
\begin{equation}
\sum_{j=1}^{\kappa_n} c_j[ \lambda(\phi_j,\phi_i) - (K\phi_j,\phi_i)] = (f,\phi_i), \hspace{1cm} i = 1, \cdots, \kappa_n. 
\end{equation}
This so-called Galerkin method with orthogonal basis raises an important question: When does the resulting sequence of approximate solutions $u_n \in V_n$ converge to $u$ in $V?$
We postpone the proof of this question to a specific nonlinear boundary value problem to the next section. For the moment, we consider some numerical examples, to make it evident that the method indeed works!
\subsection{A Galerkin Method with Trigonometric Polynomials}
We consider the problem of solving the integral equation
\begin{equation} \label{eg2.11}
\lambda u(x) - \int_0^{2\pi} k(x,y)u(y)\,dy = f(x), \hspace{1cm} 0 \leq x \leq 2\pi, \end{equation}
with $k(x,y)$ and $f(x)$ being $2\pi-$ periodic functions. We work within the framework of the Hilbert space $V = L^2(0,2\pi).$ The inner product is given by,
\[ (u,v) = \int_0^{2\pi} u(x)\overline{v(x)}\,dx. \]
We use as basis the functions 
\[ \phi_j(x) = e^{\iota jx}, \hspace{1cm} j = 0, \pm 1, \pm 2, \cdots, \pm n. \]
The orthogonal projection of $L^2(0,2\pi)$ onto $V_n$ is just the $n-$th partial sum of the series, 
\begin{equation} P_n u(x) = \frac{1}{2\pi} \sum_{j=-n}^n (u,\phi_j)\phi_j(x). \end{equation}
With respect to the basis $\phi_j$ defined above, the linear system for 
\[ (\lambda - P_nK)u_n = P_nf, \]
is given by
\begin{equation}
2\pi\lambda c_k - \sum_{j=-n}^n c_j \int_0^{2\pi} \int_0^{2\pi} e^{\iota jy -\kappa x}k(x,y) \,dy \,dx = \int_0^{2\pi} e^{-\iota \kappa x} f(x)\,dx, \hspace{1cm} \kappa = -n \cdots n. 
\end{equation} 
It can be checked that the solution $u_n$ is given by
\[ u_n(x) = \sum_{j=-n}^n c_je^{\iota j x}. \]
The integrals above are evaluated numerically, as the following numerical example shows, with slightly different orthonormal basis functions, namely orthonormal box functions. 
\begin{example} We consider the first kind Fredhom Equation (the problem has been borrowed from p.109 of \cite{Wing})
\begin{equation}g(x) = \int_0^1 ye^{-xy^2}f(y)\,dy \end{equation}
where $f$ is the unknown function, and $g(x)$ depends on the two parameters $y_1,y_2 ; y_1 <y_2,$ is given as
\begin{equation}
g(x) = \left\{ \begin{array}{cc} \frac{1}{2x}[e^{-xy_1^2}- e^{-xy_2^2}], & x > 0\\
\frac{1}{2}(y_2-y_1), & x = 0. \end{array}\right. 
\end{equation}
By direct substitution, it is easy to check that the solution is given by
\begin{equation}
f(y) = \left\{\begin{array}{cc} 1, & y_1 < y < y_2\\
0, & \mbox{ all other } y. 
\end{array}
\right.
\end{equation}
We provide a simple matlab code to obtain a Galerkin Approximation to the solution:
\begin{verbatim}
function [A,b,x] = wing(n,t1,t2)
% WING Test problem with a discontinuous solution.
%
% [A,b,x] = wing(n,t1,t2)
%
% Discretization of a first kind Fredholm integral eqaution with
% kernel K and right-hand side g given by
%    K(s,t) = t*exp(-s*t^2)                       0 < s,t < 1
%    g(s)   = (exp(-s*t1^2) - exp(-s*t2^2)/(2*s)  0 < s   < 1
% and with the solution f given by
%    f(t) = | 1  for  t1 < t < t2
%           | 0  elsewhere.
%
% Here, t1 and t2 are constants satisfying t1 < t2.  If they are
% not speficied, the values t1 = 1/3 and t2 = 2/3 are used.

% Reference: G. M. Wing, "A Primer on Integral Equations of the
% First Kind", SIAM, 1991; p. 109.

% Discretized by Galerkin method with orthonormal box functions;
% both integrations are done by the midpoint rule.


% Initialization.
if (nargin==1)
  t1 = 1/3; t2 = 2/3;
else
  if (t1 > t2), error('t1 must be smaller than t2'), end
end
A = zeros(n,n); h = 1/n;

% Set up matrix.
sti = ((1:n)-0.5)*h;
for i=1:n
  A(i,:) = h*sti.*exp(-sti(i)*sti.^2);
end

% Set up right-hand side.
if (nargout > 1)
  b = sqrt(h)*0.5*(exp(-sti*t1^2)' - exp(-sti*t2^2)')./sti';
end

% Set up solution.
if (nargout==3)
  I = find(t1 < sti & sti < t2);
  x = zeros(n,1); x(I) = sqrt(h)*ones(length(I),1);
end
\end{verbatim}
\end{example}
\section{A Nonlinear Example}
In this section, we wish to approximate the solutions to the nonlinear bvp, pass to the limit and discuss the existence of its solutions. Let $\Omega \subset \R^n, n \leq 4$ be a bounded open set. We need to discuss existence of weak solutions $u \in H^1_0(\Omega)$ of the problem, 
\begin{equation} \label{3.1}
\begin{array}{l}
-\Delta u - \lambda u + u^3 = f, \hspace{1cm} \mbox{in  } \Omega, \\
u = 0, \hspace{1cm} \mbox{ on } \Gamma. 
\end{array}
\end{equation}
where $f \in L^2(\Omega)$ and $\lambda \in \R.$ Notice that if $u \in H^1(\Omega),$ by Sobolev Inclusion Theorem, $u \in L^4(\Omega),$ and 
\[ \big| \int_\Omega u^3v \big| \leq |u|^3_{0,4,\Omega}|v|_{0,4,\Omega} \leq C \|u\|^3_{1,\Omega}\|v\|_{1,\Omega}. \]
for $ v \in H^1_0(\Omega).$ Hence $u^3 \in H^{-1}(\Omega).$
Accordingly, the weak formulation of the problem \eqref{3.1} will be to find $u \in H^1_0(\Omega)$ such that
\begin{equation} \label{3.2}
\int_\Omega \nabla u \cdot \nabla v - \lambda \int_\Omega uv + \int_\Omega u^3 v = \int_\Omega fv,
\end{equation}
for every $v \in H^1_0(\Omega).$ 
Since $H^1_0(\Omega)$ is a separable Hilbert Space, let $\{w_1,w_2, \cdots \}$ be an orthonormal basis for this space. Define
\[ W_m = span\{w_1, w_2, \cdots w_m\}.\]
We now proceed in steps. \\
\textit{Step 1:} We look for $u_m \in W_m$ such that 
\begin{equation}
\label{gal}
\int_\Omega \nabla u_m \cdot \nabla v - \lambda \int_\Omega u_mv + \int_\Omega u_m^3 v = \int_\Omega fv, \hspace{1cm} v \in W_m.
\end{equation}
It suffices now if \eqref{gal} is verified for $v = w_i, 1 \leq i \leq m.$ Let $\xi \in \R^m,$ to each such $\xi$ we associate a unique $v \in W_m$ by the map
\begin{equation} \label{3.4}
v = \sum_{i=1}^m \xi_i w_i. \end{equation}
This map is a linear bijection between $\R^m$ and $W_m,$ and more importantly, since the basis $\{w_i\}$ are orthonormal in $H^1_0(\Omega)$ 
\begin{equation}\label{star} |v|_{1,\Omega}^2 = |\xi|^2. \end{equation}
Remark that $H^1_0(\Omega)$ is equipped with the norm $|\cdot |_{1,\Omega}.$ Define now, $F : \R^m \to \R^m, $ by
\[ (F(\xi))_i = \int_\Omega \nabla v \cdot \nabla w_i - \lambda \int_\Omega vw_i + \int_\Omega v^3 w_i - \int_\Omega fw_i,\]
where $v$ is as given in \eqref{3.4}. Now \eqref{gal} has a solution if there exists a $\xi$ such that $F(\xi) = 0. $ Now, 
\begin{equation} \begin{array}{l}
(F(\xi),\xi) = \sum_{i=1}^m (F(\xi))_i \xi_i, \\
\phantom{(F(\xi),\xi)} = |v|_{1,\Omega}^2 - \lambda |v|_{0,\Omega}^2 + \int_\Omega v^4 - \int_\Omega fv\\
\phantom{(F(\xi),\xi)} \geq |v|_{1,\Omega}^2 - \lambda |v|_{0,\Omega}^2 - |f|_{0,\Omega}|v|_{0,\Omega}. 
\end{array}
\end{equation}
At this stage, in order to proceed, we need to make some restriction on $\lambda.$ Now from a course on PDE and functional analysis \cite{kesavan}, we recall that the eigenvalues of the Dirichlet Problem of the Laplace operator are characterized in terms of the Rayleigh Quotient 
\[ R(v) = \frac{|v|_{1,\Omega}^2}{|v|_{0,\Omega}^2}, \]
In particular, 
\[ |v|^2_{0,\Omega} \leq \frac{1}{\lambda_1}|v|_{1,\Omega}^2. \]
Hence we have, 
\begin{equation}
(F(\xi),\xi) \geq \left(1 - \frac{\lambda}{\lambda_1}\right) |\xi|^2 - \frac{|f|_{0,\Omega}}{\sqrt{\lambda_1}}|\xi|.
\end{equation}
Provided that $\lambda < \lambda_1, $  we can choose $|\xi| = R$ large enough so that 
\[ (F(\xi),\xi) \geq 0 \hspace{1cm} \mbox{ for } |\xi| = R. \]
Hence, by Brouwers Fixed Point Theorem, $\exists \xi^m$ such that, 
\[ |\xi^m | \leq R, \hspace{1cm} F(\xi^m) = 0. \]
At this stage, we set $u_m = \sum_{i=1}^m \xi_i^m w_i$ and $u_m \in W_m$  will be a solution to \eqref{3.3}. Furthermore, \eqref{star} implies that 
\[ |u_m|_{1,\Omega} \leq R, \] 
with $R$ depending only on $\lambda, f.$\\
\textit{Step 2:} Since $\{u_m\}$ is uniformly bounded in $H^1_0(\Omega),$ Arzela Ascoli allows us to extract a weakly convergent subsequence which we continue to denote by $\{u_m\}.$ Let 
\[ u_m \to u \hspace{1cm} \mbox{ in } H^1_0(\Omega) \mbox{ weakly }. \]\\
\textit{Step 3:} As usual, we define the space of infinitely smooth functions $C_0^\infty (\Omega ) = \mathcal{D}(\Omega).$ Let $v \in \mathcal{D}(\Omega).$ Then there exits $v_m \in W_m$ such that $v_m \to v$ strongly in $H^1_0(\Omega).$ In fact, we can choose
\[ v_m = \sum_{i=1}^m \left( \int_\Omega \nabla v \cdot \nabla w_i\right)w_i. \]
Then, by \eqref{gal}
\begin{equation}
\label{3.3}
\int_\Omega \nabla u_m \cdot \nabla v_m - \lambda \int_\Omega u_mv_m + \int_\Omega u_m^3 v_m = \int_\Omega fv_m.
\end{equation}
Now we pass to the limit, firstly for the linear terms. It is easy that $u_m \to u$ weakly in $H^1_0(\Omega)$ so strongly in $L^2(\Omega),$ $v_m \to v$ strongly in $H^1_0(\Omega)$ and $L^2(\Omega). $ Thus, 
\begin{equation} \begin{array}{l}
\lim_{m\to \infty} \int_\Omega \nabla u_m \cdot \nabla v_m = \int_\Omega \nabla u \cdot \nabla v, \\
\lim_{m\to \infty} \int_\Omega u_m v_m= \int_\Omega uv, \\
\lim_{m\to \infty} \int_\Omega fv_m = \int_\Omega fv. \\
\end{array}
\end{equation}
\textit{Step 4:} Finally we pass to the limit in the nonlinear term, 
\begin{equation} 
\int_\Omega u_m^3v_m - \int_\Omega u^3v = \int_\Omega u_m^3(v_m-v) + \int_\Omega (u_m^3 - u^3)v, = A + B \end{equation}
\begin{equation} \begin{array}{l}
|A|= 
\big|\int_\Omega u_m^3(v_m-v)\big| \leq |u_m|^3_{0,4,\Omega} |v_m-v|_{0,4,\Omega}, \\
\phantom{\big|\int_\Omega u_m^3(v_m-v)\big|} \leq C |u_m|^3_{1,\Omega} |v_m -v|_{1,\Omega}\\
\phantom{\big|\int_\Omega u_m^3(v_m-v)\big|} \leq CR^3 |v_m -v|_{1,\Omega}.\\ \\

|B|= \big|\int_\Omega (u_m^3 - u^3)v\big| = \big|\int_\Omega (u_m-u)(u_m^2 + u_mu + u^2)v \big| \\
\phantom{\big|\int_\Omega (u_m^3 - u^3)v\big| } \leq |v|_{0,\infty,\Omega} |u_m-u|_{0,\Omega}|u_m^2 + u_mu + u^2|_{0,\Omega}\\
\phantom{\big|\int_\Omega (u_m^3 - u^3)v\big| } \leq C|u_m - u|_{0,\Omega}. 

\end{array}
\end{equation}
both of which separately go to $0$ as $m \to \infty.$
Thus, for every $v \in \mathcal{D}(\Omega)$
\begin{equation} 
\int_\Omega \nabla u \cdot \nabla v - \lambda \int_\Omega uv + \int_\Omega u^3 v = \int_\Omega fv,
\end{equation}
and since $\mathcal{D}(\Omega)$ is dense in $H^1_0(\Omega)$, and the above expression is linear in $v$ we complete our discussion on existence of solutions.

\end{document}